\documentclass[12pt, reqno]{amsart}
\makeatletter
\@namedef{subjclassname@1991}{$\mathrm{1991}$ Mathematics Subject Classification}
\@namedef{subjclassname@2000}{$\mathrm{2000}$ Mathematics Subject Classification}
\@namedef{subjclassname@2010}{$\mathrm{2010}$ Mathematics Subject Classification}
\@namedef{subjclassname@2020}{$\mathrm{2020}$ Mathematics Subject Classification}
\makeatother
\usepackage{amsmath,amsthm, amscd, amsfonts, amssymb, graphicx, color}
\usepackage[bookmarksnumbered, colorlinks, plainpages,linkcolor=blue,urlcolor=blue,citecolor=blue]{hyperref}
\newtheorem{thm}{Theorem}[section]
\newtheorem{cor}[thm]{Corollary}
\newtheorem{lem}[thm]{Lemma}

\newtheorem{defn}[thm]{Definition}

\newtheorem{exam}[thm]{Example}
\numberwithin{equation}{section}

\begin{document}

\title{hybrid core-EP-$(b,c)$ inverses in rings}

\author{Huanyin Chen}
\address{School of Big Data, Fuzhou University of International Studies and Trade, Fuzhou 350202, China}
\email{<huanyinchenfz@163.com>}


\begin{abstract}In this paper, we study when the core inverse in a ring coincides with its hybrid $(b,c)$-inverse, i.e., hybrid core-$(b,c)$-inverse.
We present many characterizations of hybrid core-$(b,c)$-inverse. To establish a broader framework for this generalized inverses, we define hybrid core-EP-$(b,c)$-inverse that serves as a natural extension of the hybrid core-$(b,c)$-inverse. We characterize this new generalized inverse by combining the hybrid core-$(b,c)$-inverses and quasinilpotents. This generalized inverse is thereby examined through a novel limit-based approach. Its polar-like properties and image-based representations are presented.\end{abstract}

\maketitle

\section{Introduction}

A ring $R$ is called a *-ring if there exists an involution $*: x\to x^*$ satisfying $(x+y)^*=x^*+y^*, (\lambda x)^*=\overline{\lambda} x^*, (xy)^*=y^*x^*, (x^*)^*=x$. An element $a$ in a *-ring $R$ has core inverse if and only if there exists an element $x\in R$ such that $$ax^2=x, (ax)^*=ax, xa^2=a.$$ If such $x$ exists, it is unique, and denote it by $a^{\tiny\textcircled{\#}}$ (see~\cite{M,X1}). A square complex matrix $A$ has core inverse $X$ if and only if $AX=P_A, \mathcal{R}(X)\subseteq \mathcal{R}(A)$ (see~\cite{BT}). Here, $\mathcal{R}(X)$ and $\mathcal{R}(A)$ are the range spaces of $X$ and $A$, respectively, and $P_A$ is the orthogonal projection onto $R(A)$.

Gao and Chen extended the concept of the core inverse and introduced the notion of core-EP inverse (i.e., pseudo core inverse) (see~\cite{GC,Z1}). An element $a\in R$ has core-EP inverse if there exist $x\in R$ and $k\in \Bbb{N}$ such that $$ax^2=x, (ax)^*=ax, xa^{k+1}=a^k.$$ If such $x$ exists, it is unique, and denote it by $a^{\tiny\textcircled{D}}$. A square complex matrix $A$ has core-EP inverse $X$ if and only if $X=XAX, \mathcal{R}(X)=\mathcal{R}(X^*)=\mathcal{A}^k$, where $k=ind(A)$.

Following Drazin~\cite{D1}, an element $a\in R$ has left hybrid $(b,c)$-inverse if there exists an element $x\in R$ such that
$$xax=x, \ell(x)=\ell(b)~\mbox{and}~Rx=Rc.$$ If such $x$ exists, it is unique and is denoted by $a_{lh}^{(b,c)}.$ Dually, an element $a\in R$ has right hybrid $(b,c)$-inverse if there exists an element $x\in R$ such that
$$xax=x, xR=bR~\mbox{and}~r(x)=r(c).$$ If such $x$ exists, it is unique and is denoted by $a_{rh}^{(b,c)}.$

The purpose of this paper is to investigate when the core inverse of a ring element coincides with its left (right) hybrid $(b,c)$-inverse. We adopt

\begin{defn} An element $a\in R$ has left hybrid core-$(b,c)$-inverse if $a\in R^{\tiny\textcircled{\#}}$ and $a^{\tiny\textcircled{\#}}=a_{lh}^{(b,c)}$. We use $a_{lh(b,c)}^{\tiny\textcircled{\#}}$ to stands for $a^{\tiny\textcircled{\#}}$. The set of all left hybrid core-$(b,c)$ invertible elements in $R$ is denoted by $R_{lh(b,c)}^{\tiny\textcircled{\#}}$.\end{defn}

Dually, we can define the right hybrid core-$(b,c)$-inverse by the existence of $y\in R$ such that
 $a\in R^{\tiny\textcircled{\#}}$ and $a^{\tiny\textcircled{\#}}=a_{rh}^{(b,c)}$.
We use $a_{rh(b,c)}^{\tiny\textcircled{\#}}$ to stand for $a^{\tiny\textcircled{\#}}$. The set of all right hybrid core-$(b,c)$ invertible elements in $R$ is denoted by $R_{rh(b,c)}^{\tiny\textcircled{\#}}$.

To establish a broader framework for this generalized inverses, we define hybrid core-EP-$(b,c)$-inverse that serves as a natural extension of the hybrid core-$(b,c)$-inverse.

\begin{defn} An element $a\in R$ has left hybrid core-EP-$(b,c)$-inverse if there exist $x,y\in R$ such that $$a=x+y, x^*y=yx=0, x\in R_{lh(b,c)}^{\tiny\textcircled{\#}}, y\in R~\mbox{is nilpotent}.$$ We use $a_{lh(b,c)}^{\tiny\textcircled{D}}$ to stands for $x_{lh(b,c)}^{\tiny\textcircled{\#}}$. The set of all generalized left hybrid core-EP-$(b,c)$-invertible elements in $R$ is denoted by $R_{lh(b,c)}^{\tiny\textcircled{D}}$.\end{defn}

In Section 2, we characterize the left (right) hybrid core-$(b,c)$-inverse. We prove that $a\in R_{lh(b,c)}^{\tiny\textcircled{\#}}$ if and only if
$a\in R^{\tiny\textcircled{\#}}$ and $\ell(a)=\ell(b), Ra^*=Rc$, if and only if
$a\in R_{lh}^{(b,c)}$ and $\ell(a)=\ell(b), Ra^*=Rc.$

An element $a$ in $\mathcal{A}$ is EP (i.e., an EP element) if there exists some $x\in \mathcal{A}$ such that
$ax^2=x, (ax)^*=xa, a=xa^2.$ Evidently, $a\in \mathcal{A}$ is EP if and only if there exists $x\in \mathcal{A}$ such that $a^2x=a, ax=xa, (ax)^*=ax$ if and only if there exists $x\in \mathcal{A}$ such that $ax^2=x, (xa)^*=xa, xa^2=a$ if and only if $a\in \mathcal{A}^{\#}$ and $(aa^{\#})^*=aa^{\#}$ (see~\cite{CM,MP,MDK,XS1}). In Section 3, we characterize characterizations of left (right) hybrid core-$(b,c)$-inverse via the EP property.

In Section 4, the left (right) hybrid core-EP-$(b,c)$-inverse is examined through a novel limit-based approach.
We prove that $a\in R_{lh(b,c)}^{\tiny\textcircled{D}}$ if and only if $a\in R^{\tiny\textcircled{D}}$ and $aa^{\tiny\textcircled{D}}a\in R_{lh(b,c)}^{\tiny\textcircled{\#}}.$ The polar-like property of left (right) hybrid core-EP-$(b,c)$-inverse are presented.

Finally, in Section 5, we are concerned with image-based representations for the left (right) hybrid core-EP-$(b,c)$-inverse. We prove that
$a\in R_{lh(b,c)}^{\tiny\textcircled{D}}$ if and only if $a^m\in R_{lh(b,c)}^{\tiny\textcircled{\#}}$ for some $m\in {\Bbb N}$.

Throughout the paper, all *-ring are associative with an identity. Let ${\Bbb C}^{n\times n}$ be the ring of all $n\times n$ complex matrices with conjugate transpose $*$. $\ell(x)$ and $r(x)$ stand for the left and right annihilators of $x\in R$, respectively. The set of all nilpotent elements in $R$ is denoted by $R^{nil}$.

\section{hybrid core-$(b,c)$-inverse}

The purpose of this section is to investigate elementary properties of the left (right) hybrid core-$(b,c)$-inverse in a *-ring. Our starting points is the following.

\begin{thm} Let $a,b,c\in R$. Then the following are equivalent:\end{thm}
\begin{enumerate}
\item [(1)] $a\in R_{lh(b,c)}^{\tiny\textcircled{\#}}$.
\item [(2)] $a\in R^{\tiny\textcircled{\#}}$ and $\ell(a)=\ell(b), Ra^*=Rc.$
\item [(3)] $a\in R_{lh}^{(b,c)}$ and $\ell(a)=\ell(b), Ra^*=Rc.$
\end{enumerate}
\begin{proof} $(1)\Rightarrow (2)$ By hypothesis, $a\in R^{\tiny\textcircled{\#}}$ and $x:=a^{\tiny\textcircled{\#}}=a_{lh}^{(b,c)}$.
Then $a=xa^2$ and $x=ax^2$; hence, $xR=aR$. Thus, $\ell(b)=\ell(x)=\ell(a)$. Since $x=xax$ and $(ax)^*=ax$, we see that
$x^*=(ax)x^*$. Since $a=axa=(ax)^*a=x^*(a^*a)$. This implies that $x=xx^*a^*$ and $a^*=(a^*a)x$. Hence,
$Ra^*=Rx$. This implies that $Ra^*=Rx=Rc$.

$(2)\Rightarrow (3)$  By hypothesis, $a\in R^{\tiny\textcircled{\#}}$. Set $x=a^{\tiny\textcircled{\#}}$.
As in the preceding discussion, $xax=x, \ell(x)=\ell(a)$ and $Ra^*=Rx$.
Thus we see that $$xax=x, \ell(x)=\ell(b), Rx=Rc.$$
Therefore $x=a_{lh}^{(b,c)}$, as required.

$(3)\Rightarrow (1)$ Let $x=a_{lh}^{(b,c)}$. Then $$xax=x, \ell(x)=\ell(b), Rx=Rc.$$
By hypothesis, we have $$xax=x, \ell(x)=\ell(a), Rx=Ra^*.$$ Since $(1-xa)x=0$, we have $1-xa\in \ell(a)$, and so $a=xa^2$.
On the other hand, $x(1-ax)=0$; whence, $a^*(1-ax)=0$. Thus $a^*=a^*ax$.
Hence, $(ax)^*=(ax)^*ax$; and so $(ax)^*=ax$. As $x=xax$, we have $(1-xa)a=0$, and then $(1-xa)x=0$.
Thus $a=axa$, and then $1-ax\in \ell(a)$. Accordingly, $1-ax\in\ell(x)$, and so $(1-ax)x=0$, that is, $x=ax^2$.
This implies that $a\in R^{\tiny\textcircled{\#}}$ and
$a^{\tiny\textcircled{\#}}=x=a_{lh}^{(b,c)}$, as asserted.\end{proof}

We denote $x$ in Theorem 2.1 by $a_{lh(b,c)}^{\tiny\textcircled{\#}}$, and call it the left hybrid core-$(b,c)$-inverse of $a$. As an immediate consequence, we prove that the core inverse $a^{\tiny\textcircled{\#}}$ of $a$ can be characterized as its left hybrid core-$(b,c)$-inverse, namely with $(b,c)=(a,a^*)$ .

\begin{cor} Let $a,b,c\in R$. Then the following are equivalent:\end{cor}
\begin{enumerate}
\item [(1)] $a\in R_{lh(b,c)}^{\tiny\textcircled{\#}}$.
\item [(2)] $a\in R_{lh}^{(b,c)}$ and $\ell(a)=\ell(b), r(a^*)=r(c).$
\end{enumerate}
\begin{proof} $(1)\Rightarrow (2)$ By virtue of Theorem 2.1, we have $a\in R_{lh}^{(b,c)}$ and $\ell(a)=\ell(b), Ra^*=Rc.$
Therefore $\ell(a)=\ell(b), r(a^*)=r(c).$

$(2)\Leftarrow (2)$ Set $x=a_{lh}^{(b,c)}$. Then $$xax=x, \ell(x)=\ell(b)~\mbox{and}~Rx=Rc.$$ Thus $(1-xa)x=0$, and so $(1-xa)b=0$, that is,
$1-xa\in \ell(b)$. By hypothesis, $1-xa\in \ell(a)$. This implies that $a=xa^2$. Moreover, $1-ax\in r(x)$, and then $1-ax\in r(c)$.
We infer that $a^*(1-ax)=0$, and then $a^*=a^*(ax)$. It follows that $a=(ax)^*a$; hence, $ax=(ax)^*ax$. Therefore $(ax)^*=ax$.

As $a^*=a^*(ax)$, we have $a=(ax)^*a=axa$. Hence, $(1-ax)a=0$, and then $1-ax\in \ell(a)\subseteq \ell(b)$. Since $\ell(x)=\ell(b)$, we get $(1-ax)x=0$.
Therefore $x=ax^2$, and so $a^{\tiny\textcircled{\#}}=x$, as required.\end{proof}

\begin{cor} Let $a,b,c\in R$. Then the following are equivalent:\end{cor}
\begin{enumerate}
\item [(1)] $a\in R_{lh(b,c)}^{\tiny\textcircled{\#}}$.
\vspace{-.5mm}
\item [(2)] $Rb=Rcab, \ell(cab)=\ell(c), \ell(a)=\ell(b), r(a^*)=r(c).$
\vspace{-.5mm}
\item [(3)] $Rca\bigoplus \ell(b)=R, Rc\bigcap \ell(ab)=0, \ell(a)=\ell(b), r(a^*)=r(c).$
\end{enumerate}
\begin{proof} $(1)\Rightarrow (2)$ In view of Corollary 2.2, $a\in R_{rh}^{(b,c)}, \ell(a)=\ell(b), r(a^*)=r(c).$
Set $z=_{lh(b,c)}^{\tiny\textcircled{\#}}$. Then $$zaz=z, \ell(z)=\ell(b), Rz=Rc.$$ Hence,
$1-za\in \ell(z)=\ell(b)$. Thus $b=zab$. Write $z=rc$ for some $r\in R$. Then $b=r(cab)$, and so
$Rb\subseteq Rcab$.

If $scab=0$ for some $s\in R$, then $sca\in \ell(b)=\ell(z)$, and so $scaz=0$.
As $z(1-az)=0$, we see that $c(1-az)=0$; hence, $c=caz$. This implies that
$sc=0$, and so $s\in \ell(c)$. Then $\ell(cab)\subseteq \ell(c)$. Obviously, $cab\in Rb$ and $\ell(c)\subseteq \ell(cab)$. Hence $Rb=Rcab, \ell(cab)=\ell(c)$.

$(2)\Rightarrow (3)$ Since $b\in Rcab$, we can find $z\in R$ such that
$b=zcab$, and then $(1-zca)b=0$. Since $1=zca+(1-zca)$, we deduce that
$Rca+\ell(b)=R$.

If $r\in Rc\bigcap \ell(ab)$, then $r=sc$ for some $s\in R$.
Then $scab=rab=0$. This implies that $s\in \ell(cab)=\ell(c)$; hence, $r=sc=0$.
Thus $Rc\bigcap \ell(ab)=0$.

If $r\in Rca\bigcap \ell(b)$, then $r=xca$ for some $x\in R$. Hence
$xcab=0$, and so $xc\in Rc\bigcap \ell(ab)=0$. Thus $r=(xc)a=0$. This implies that
$Rca\bigcap \ell(b)=0$, and then $Rca\bigoplus \ell(b)=R$.

$(3)\Rightarrow (1)$ By hypothesis, we have $Rca+\ell(b)=R$ and $Rc\bigcap \ell(ab)=0$.
Then $xca+y=1$ for some $x\in R, y\in \ell(b)$.
Hence $b=xcab$, and so $cab=caxcab$. This implies that $c-caxc\in Rc\bigcap \ell(ab)=0$. Hence, $c=caxc$.
Set $z=xc$. Then $zaz=x(caxc)=xc=z$. Clearly, $Rz\subseteq Rc\subseteq Rz$. Hence, $Rz=Rc$.

Since $b=zab$, we have $\ell(z)\subseteq \ell(b)$. If $r\in \ell(b)$, then $r\in \ell(a)$.
Hence $rz=rxc\in Rc\bigcap \ell(ab)=0$, i.e., $r\in \ell(z)$. Thus $\ell(z)=\ell(b)$.
Therefore $a\in R_{lh}^{(b,c)}$. This completes the proof by Corollary 2.2.
\end{proof}

From the dual perspective, we characterize the right hybrid core-$(b,c)$-inverse as follows.

\begin{thm} Let $a,b,c\in R$. Then the following are equivalent:\end{thm}
\begin{enumerate}
\item [(1)] $a\in R_{rh(b,c)}^{\tiny\textcircled{\#}}$.
\item [(2)] $a\in R^{\tiny\textcircled{\#}}$ and $aR=bR, r(a^*)=r(c).$
\item [(3)] $a\in R_{rh}^{(b,c)}$ and $aR=bR, r(a^*)=r(c).$
\item [(4)] $a\in R_{rh}^{(b,c)}$ and $\ell(a)=\ell(b), r(a^*)=r(c).$
\end{enumerate}

\begin{cor} Let $a,b,c\in R$. Then the following are equivalent:\end{cor}
\begin{enumerate}
\item [(1)] $a\in R_{rh(b,c)}^{\tiny\textcircled{\#}}$.
\vspace{-.5mm}
\item [(2)] $cR=cabR, r(cab)=r(b), \ell(a)=\ell(b), r(a^*)=r(c).$
\vspace{-.5mm}
\item [(3)] $abR+r(a^*)=R, bR\bigcap r(ca)=0, \ell(a)=\ell(b), r(a^*)=r(c).$
\end{enumerate}

\begin{cor} Let $a,b,c\in R$. Then the following are equivalent:\end{cor}
\begin{enumerate}
\item [(1)] $a\in R_{b,c}^{\tiny\textcircled{\#}}$.
\item [(2)] $a\in R_{rh}^{(b,c)}\bigcap R_{lh}^{(b,c)}.$
\end{enumerate}
\begin{proof} $(1)\Rightarrow (2)$ In light of~\cite[Theorem 2.1]{CM1}, we have $aR=bR, r(a^*)=r(c).$ It follows from Theorem 2.4 that
$a\in R_{rh}^{(b,c)}$. Dually, $a\in R_{lh}^{(b,c)}$ by Theorem 2.1.

$(2)\Rightarrow (1)$ In view of Theorem 2.4, $aR=bR$. By virtue of Theorem 2.1, $Ra^*=Rc$.
According to ~\cite[Theorem 2.1]{CM1}, $a\in R_{b,c}^{\tiny\textcircled{\#}}$, as asserted.\end{proof}

\begin{exam}\end{exam} Let  $R={\Bbb Z}$ be the ring of all integers with the identical map as an involution.
Take $a=1, b=1, c=3\in R$, then $a$ is right hybrid core-$(1,2)$-invertible. Indeed, $x=1$ is the right hybrid core-$(1,2)$-inverse of $a$ as
 $a$ has core inverse $x$, $aR=bR, \ell(a^*)=\ell(c)$. However, it has not left hybrid core $(b,c)$-inverse as $Ra^*\neq Rc$.\\

We next consider the special case $b=c$ for the left hybrid core-$(b,c)$-inverse.

\begin{thm} Let $a,b\in R$. Then the following are equivalent:\end{thm}
\begin{enumerate}
\item [(1)] $a\in R_{lh(b,b)}^{\tiny\textcircled{\#}}$.
\item [(2)] $a\in R_{(b,b)}^{\tiny\textcircled{\#}}$.
\item [(3)] $a,b\in R^{\tiny\textcircled{\#}}$ and $aa^{\tiny\textcircled{\#}}=b^{\tiny\textcircled{\#}}b.$
\end{enumerate}
\begin{proof} $(1)\Rightarrow (2)$ In view of Theorem 2.1, $a\in R^{\tiny\textcircled{\#}}$ and $\ell(a)=\ell(b), Ra^*=Rb.$
Then $a\in R$ is regular. This implies that $a^*$ is regular. Since $Ra^*=Rb,$ we see that $b\in R$ is regular.
Write $a=aa^{-}a$ and $b=bb^{-}b$. Since $(1-aa^{-})a=0$, we have $(1-aa^{-})b=0$. Thus, $b=aa^{-}b\in aR$.
As $(1-bb^{-})b=0$, we have $(1-bb^{-})a=0$; hence, $a=bb^{-}a\in bR$. Therefore $aR=bR$. According to~\cite[Theorem 2.1]{CM1}, $a\in R_{(b,b)}^{\tiny\textcircled{\#}}$

$(2)\Rightarrow (1)$ Since $a\in R_{(b,b)}^{\tiny\textcircled{\#}}$, by virtue of ~\cite[Theorem 2.1]{CM1},
$a\in R^{\tiny\textcircled{\#}}$ and $aR=bR, Ra^*=Rb.$ Therefore $\ell(a)=\ell(b)$.
In light of Theorem 2.1, $a\in R_{lh(b,b)}^{\tiny\textcircled{\#}}$.

$(2)\Leftrightarrow (3)$ This is proved in ~\cite[Theorem 2.6]{CM1}.\end{proof}

\begin{cor} Let $a,b\in R$. Then the following are equivalent:\end{cor}
\begin{enumerate}
\item [(1)] $a\in R_{lh(b,b)}^{\tiny\textcircled{\#}}$.
\item [(2)] $a,b,ba\in R^{\tiny\textcircled{\#}}$, $a^{\tiny\textcircled{\#}}=(ba)^{\tiny\textcircled{\#}}b, b^{\tiny\textcircled{\#}}=a(ba)^{\tiny\textcircled{\#}}.$
\item [(3)] $a,b,ba\in R^{\tiny\textcircled{\#}}$, $aR=bR, (ba)^{\tiny\textcircled{\#}}=
a^{\tiny\textcircled{\#}}b^{\tiny\textcircled{\#}}.$
\end{enumerate}
\begin{proof} Straightforward by Theorem 2.8 and ~\cite[Theorem 2.8]{CM1}.\end{proof}

Let $\mathcal{N}(X)$ denote the null space of a matrix $X$. We now derive

\begin{cor} Let $A,B\in {\Bbb C}^{n\times n}$. Then the following are equivalent:\end{cor}
\begin{enumerate}
\item [(1)] $A$ has left hybrid core-$(B,B)$-inverse.
\vspace{-.5mm}
\item [(2)] $A$ has core-$(B,B)$-inverse.
\vspace{-.5mm}
\item [(3)] $rank(B)=rank(BAB), \mathcal{N}(A^*)=\mathcal{N}(B^*)=\mathcal{N}(B).$
\end{enumerate}
\begin{proof} $(1)\Leftrightarrow (2)$ This is proved by Theorem 2.8.

$(2)\Leftrightarrow (3)$ By virtue of ~\cite[Theorem 1.5]{R}, $A$ has hybrid $(B,B)$-inverse if and only if $rank(B)=rank(BAB)$, and we are through by Theorem 2.1.\end{proof}

\section{EP properties}

The aim of this section is to characterize the hybrid core-$(b,c)$-inverse of an element by exploiting the
EP property of certain elements. We derive necessary and sufficient conditions for the existence of such generalized inverses.

\begin{thm} Let $a,b,c\in R$ with $\ell(b)=\ell(c)$. Then $a\in R_{lh(b,c)}^{\tiny\textcircled{\#}}$ if and only if\end{thm}
\begin{enumerate}
\item [(1)] $(ac)^2$ is EP;
\item [(2)] $\ell(ca)=\ell(b), (ac)^{\pi}c=(ca)^{\pi}a=0$.
\end{enumerate}
\begin{proof} $(1)\Rightarrow (2)$ By hypothesis, $a^{\tiny\textcircled{\#}}=a_{lh}^{(b,c)}$. In view of Corollary 2.3, we have
$Rca\bigoplus \ell(b)=R, Rc\bigcap \ell(ab)=0, \ell(a)=\ell(b), r(a^*)=r(c).$ If $rca=0$, then $rc\in Rc\bigcap \ell(ab)=0$. Hence $r\in \ell(c)$, and so $r\in \ell(b)$. Thus $\ell(ca)\subseteq \ell(b)$. If $rb=0$, then $rca=(rc)a=0$. Thus, $\ell(ca)=\ell(b)$.

Write $1=zca+y$ for some $z\in R, y\in \ell(b)$. Then $y\in \ell(c)$, and so $yca=0$.
Thus, $ca=z(ca)^2+yca=z(ca)^2$. Moreover, we have $1=(z^2ca)(ca)+y$. Set $x=z^2ca$. Then $1=xca+y$.
This implies that $(ca)^2=(ca)x(ca)^2$, and then $ca-(ca)x(ca)\in \ell(ca)\bigcap Rca=\ell(b)\bigcap Rca=0$.
Hence, $ca=(ca)x(ca)$, and so $(ca)^2=(ca)^2x(ca)$. We infer that $ca-(ca)^2x\in \ell(ca)=\ell(b).$ As $x\in Rca$, we have
$ca-(ca)^2x\in Rca\bigcap \ell(b)=0$; whence $ca=(ca)^2x$. Then $ca\in (ca)^2R\bigcap R(ca)^2$, and so $ca\in R^{\#}$. In view of Cline's formula, $ac\in R^D$, and then $(ac)^2\in R^D$. Clearly, $\big((ac)^2\big)^D\big((ac)^2\big)^2=[(ac)^D(ac)^2]^2=[a((ca)^{\#})^2c(ac)^2]^2=[a((ca)^{\#})^2(ca)^2c]^2=(ac)^2$.
Then $a^{\tiny\textcircled{\#}}=(ca)^{\#}c$. This implies that $aa^{\tiny\textcircled{\#}}=a(ca)^{\#}c$. Thus, $(ac)^2[(ac)^2]^D=ac(ac)^D=aca[(ca)^{\#}]^2c=a(ca)^{\#}c$ is a projection.
Therefore $(ac)^2\in R$ is EP.

Since $(1-aa^{\tiny\textcircled{\#}})a=0$ and $\ell(a)=\ell(b)$, we have $(1-aa^{\tiny\textcircled{\#}})b=0$. This implies that
$(1-a^{\tiny\textcircled{\#}}a)b=0$; hence, $(ca)^{\pi}b=[1-(ca)^{\#}ca]b=(1-a^{\tiny\textcircled{\#}}a)b=0$.
Since $\ell(a)=\ell(b)$, we deduce that $(ca)^{\pi}a=0$.

Furthermore, we verify that
$(ac)^{\pi}b=(1-(ac)^{\#}ac)b=(1-a((ca)^D)^2cac]b=(1-a((ca)^{\#}c))b=(1-aa^{\tiny\textcircled{\#}})b=0$.
As $\ell(b)=\ell(c)$, we deduce that $(ac)^{\pi}c=0,$ as required.

$(2)\Rightarrow (1)$ Since $(ac)^2$ is EP, $(ac)^2$ has group inverse, and then $ac\in R^{D}$ and $(ac)^2((ac)^2)^D=ac(ac)^D$.
Hence $[ac(ac)^D]^*=ac(ac)^D$. By virtue of Cline's formula, we have $ca\in R^D$.

Since $(ca)^{\pi}c=0$, we deduce that $ca=(ca)^D(ca)^2$. Then $ca\in R^{\#}$. By hypothesis,
$\ell(ca)=\ell(b)=\ell(c)$. Set $x=(ca)^{\#}c$.

Claim 1. $x=xax$. Clearly, $xax=(ca)^{\#}ca(ca)^{\#}c=(ca)^{\#}c=x$.

Claim 2. $\ell(x)=\ell(b)$. If $rx=0$, then $r(ca)^{\#}c=0$, and so $rca=0$. This implies that $rb=0$.
If $rb=0$, then $rca=0$; hence, $rx=r(ca)^{\#}c=0$, as required.

Claim 3. $Rx=Rc$. Clearly, $x=(ca)^{\#}c\in Rc$. Since $[1-(ca)^{\#}ca]ca=0$ and $\ell(ca)=\ell(c)$, we have
$[1-(ca)^{\#}ca]c=0$; whence, $c=[(ca)^{\#}c]ac=ca[(ca)^{\#}c]=cax\in Rx,$ as desired.

Accordingly, $a\in R_{lh}^{(b,c)}$ and $a_{lh}^{(b,c)}=(ca)^{\#}c$.

Set $x=(ca)^{\#}c$. Then we verify that
$$\begin{array}{rll}
ax&=&a(ca)^{\#}c=ac((ac)^D)^2ac=ac(ac)^D,\\
(ax)^*&=&ax,\\
ax^2&=&ac(ac)^D(ca)^{\#}c=(1-(ac)^{\pi}£©(ca)[(ca)^{\#}]^2c\\
&=&(ca)^{\#}c=x,\\
xa^2&=&((ca)^{\#}ca£©a=[1-(ca)^{\pi}]a=a.
\end{array}$$ Thus $a^{\tiny\textcircled{\#}}=x=a_{lh}^{(b,c)}$. Therefore $a\in R_{lh(b,c)}^{\tiny\textcircled{\#}}$, as asserted.\end{proof}

\begin{cor} Let $a,b,c\in R$ with $\ell(b)=\ell(c)$. If $a\in R_{lh(b,c)}^{\tiny\textcircled{\#}}$, then $a(ca)^{\pi}=0$.\end{cor}
\begin{proof} In view of Theorem 3.1, $ac\in R^D$, and then $$\begin{array}{rll}
a(ca)^{\pi}&=&a-a(ca)(ca)^{\#}=a-a(ca)(c((ac)^D)^2a)\\
&=&a-(ac)(ac)^Da=a-a(c(ac)^D)a\\
&=&a-a(c((ac)^D)^2a)ca=a-a((ca)^{\#}c)a\\
&=&a-aa^{\tiny\textcircled{\#}}a=0,
\end{array}$$ as required.\end{proof}

We next characterize the right hybrid core-$(b,c)$-inverse of an element $a$ using the EP property of $ab$. Surprisingly, as shown below, the left and right hybrid core-$(b,c)$-inverses are not always dual.

\begin{thm} Let $a,b,c\in R$ with $r(b)=r(c)$. Then $a\in R_{rh(b,c)}^{\tiny\textcircled{\#}}$ if and only if\end{thm}
\begin{enumerate}
\item [(1)] $ab\in R$ is EP;
\item [(2)] $r(ab)=r(b), (ab)^{\pi}b=0$ and $(ba)^{\pi}a=0$.
\end{enumerate}
\begin{proof} $\Longrightarrow $ By hypothesis, $a^{\tiny\textcircled{\#}}=a_{rh}^{(b,c)}$. In light of ~\cite[Lemma 2.2]{Z2},
$a_{rh}^{(b,c)}ab=b$, and then $r(ab)=r(b)=r(c)$. By virtue of ~\cite[Theorem 2.8]{Z}, $ab\in R^{\#}$ and $a_{rh}^{(b,c)}=b(ab)^{\#}$.
Then $a^{\tiny\textcircled{\#}}=b(ab)^{\#}$. This implies that $ab(ab)^{\#}=aa^{\tiny\textcircled{\#}}$. Thus, $[ab(ab)^{\#}]^*=ab(ab)^{\#}$.
By virtue of ~\cite[Lemma 2.1]{XS1}, $ab\in R$ is EP.

In view of Theorem 2.4, $aR=bR$, and so $a=by$ and $b=az$ for some $y,z\in R$.
Then we verify that
$$\begin{array}{rll}
(ba)^{\pi}a&=&(ba)^{\pi}by=(ba)^{\pi}(a_{rh}^{(b,c)}ab)y=(ba)^{\pi}(b(ab)^{\#})aby\\
&=&(ba)^{\pi}(ba)^d(ba)b(ab)^{\#}aby=0,\\
(ab)^{\pi}b&=&(ab)^{\pi}az=(ab)^{\pi}(aa^{\tiny\textcircled{\#}}a)z=(ab)^{\pi}a(b(ab)^{\#})az\\
&=&\big((ab)^{\pi}ab\big)(ab)^{\#}az=0,
\end{array}$$ as desired.

$\Longleftarrow $ Since $ab\in R$ is EP, it follows by~\cite[Lemma 2.1]{XS1} that $ab\in R^{\#}$ and $\big(ab(ab)^{\#}\big)^*=ab(ab)^{\#}$.
By hypothesis, $r(ab)=r(b)=r(c)$. In view of~\cite[Theorem 2.8]{Z}, $a\in R_{rh}^{(b,c)}$ and $a_{rh}^{(b,c)}=b(ab)^{\#}$.
Since $(ab)^{\pi}b=0$, we deduce that $$a(b(ab)^{\#})^2=ab(ab)^{\#}((ab)(ab)^{\#}b)(ab)^{\#}=((ab)(ab)^{\#}b)(ab)^{\#}=b(ab)^{\#}.$$ As $(ba)^{\pi}a=0$, we verify that $$b(ab)^{\#}a^2=ba((ba)^d)^2ba^2=ba(ba)^da=(1-(ba)^{\pi})a=a.$$ Therefore $a^{\tiny\textcircled{\#}}=b(ab)^{\#}=a_{rh}^{(b,c)}$, as required.\end{proof}

\begin{cor} Let $a,b\in R$. Then the following are equivalent:\end{cor}
\begin{enumerate}
\item [(1)] $a\in R_{(b,b)}^{\tiny\textcircled{\#}}$.
\item [(2)] $a\in R_{lh(b,b)}^{\tiny\textcircled{\#}}$.
\item [(3)] $a\in R_{rh(b,b)}^{\tiny\textcircled{\#}}$.
\item [(4)] $ab\in R$ is EP, $\ell(ba)=\ell(b), (ab)^{\pi}b=(ba)^{\pi}a=0$.
\item [(5)] $ab\in R$ is EP, $r(ab)=r(b), (ba)^{\pi}a=0$ and $(ab)^{\pi}b=0$.
\end{enumerate}
\begin{proof} This is obvious by choosing $b=c$ in Theorem 2.8, Theorem 3.1 and Theorem 3.3.\end{proof}

We subsequently characterize EP elements in a ring via the left (right) hybrid core-$(b,c)$-invertibility.

\begin{lem} Let $a\in R$ and $n\in {\Bbb N}$. Then the following are equivalent:\end{lem}
\begin{enumerate}
\item [(1)] $a\in R$ is EP.
\item [(2)] $a\in R$ has core-$(a^*,a^*)$ inverse.
\item [(3)] $a\in R$ has left (right) hybrid core-$(a^*,a^*)$ inverse.
\item [(4)] $a\in R$ has left (right) hybrid core-$\big((a^n)^*,a^n\big)$ inverse.
\end{enumerate}
\begin{proof} $(1)\Leftrightarrow (2)$ In view of ~\cite{DM}, $a^{\dag}$ is exact the $(a^*,a^*)$-inverse of $a$.
This equivalence is proved by ~\cite[Lemma 3.5]{XS1}.

$(2)\Rightarrow (3)$ This is proved by Corollary 3.4.

$(3)\Rightarrow (1)$ Case 1. $a\in R$ has left hybrid core-$\big(a^*,a^*\big)$-inverse. According to Theorem 2.1, $a\in R^{\tiny\textcircled{\#}}$ and $\ell(a)=\ell(a^*)$. Since $a=aa^{\tiny\textcircled{\#}}a$, we see that $1-a^{\tiny\textcircled{\#}}a\in \ell(a^*)$. Hence,
$a^*=a^{\tiny\textcircled{\#}}aa^*=a(a^{\tiny\textcircled{\#}})^2aa^*\in aR$. This implies that $a^*R\subseteq aR$.
Therefore $a\in R$ is EP by ~\cite[Theorem 3.9]{XS1}.

Case 2. $a\in R$ has right hybrid core-$\big(a^*,a^*\big)$-inverse. In view of Theorem 2.1, $a\in R^{\tiny\textcircled{\#}}$ and $aR=a^*R$. By virtue of ~\cite[Corollary 3.10]{XS1}, $a\in R$ is EP, as required.

$(1)\Rightarrow (4)$ Since $a\in R$ is EP, we see that $a^{\tiny\textcircled{\#}}=a^{\#}$.
If $ra=0$, then $r(a^n)^*=r(a^naa^{\#})^*=raa^{\#}(a^n)^*=0$; hence, $\ell(a)\subseteq \ell(a^n)^*$.
If $r(a^n)^*=0$, then $r(a^n)(a^{\#})^*=0$, and so $raa^{\tiny\textcircled{\#}}=0$. Thus $ra=()a=0.$
This implies that $\ell(a^n)^*\subseteq \ell(a)$. Hence, $\ell(a)=\ell(a^n)^*$.
 Obviously, $a^*=(a^2a^{\tiny\textcircled{\#}})^*=a^*(aa^{\tiny\textcircled{\#}})=a^*(a^{\tiny\textcircled{\#}})^na^n$.
 On the other hand, $a^n=a^n(aa^{\#})=a^n(aa^{\#})^*=a^n(a^{\#})^*a^*$.
 We infer that $Ra^*=Ra^n$. Therefore $a^{\tiny\textcircled{\#}}=a^{\big((a^n)^*,a^n\big)}$ by Theorem 2.1.

$(4)\Rightarrow (1)$ Case 1. $a\in R$ has left hybrid core-$\big((a^n)^*,a^n\big)$ inverse. According to Theorem 2.1, $a\in R^{\tiny\textcircled{\#}}$ and $Ra^*=Ra^n$. Hence, $Ra^*\subseteq Ra$. Therefore $a\in R$ is EP by~\cite[Corollary 3.10]{XS1}.

Case 2. $a\in R$ has left hybrid core-$\big((a^n)^*,a^n\big)$ inverse. According to Theorem 2..1, $a\in R^{\tiny\textcircled{\#}}$ and $aR=(a^n)^*R\subseteq a^*R$. In light of ~\cite[Corollary 3.10]{XS1}, $a\in R$ is EP, as desired.
\end{proof}

\begin{thm} Let $a,b,c\in R$. Then the following are equivalent:\end{thm}
\begin{enumerate}
\item [(1)] $a\in R$ is EP.
\item [(2)] $a\in R$ has core-$\big((a^D)^*,(a^D)^*\big)$ inverse.
\item [(3)] $a\in R$ has left (right) hybrid core-$\big((a^D)^*,(a^D)^*\big)$ inverse.
\item [(4)] $a\in R$ has left (right) hybrid core-$\big(a^D,(a^D)^*\big)$ inverse.
\item [(5)] $a\in R$ has left (right) hybrid core-$\big((a^D)^*,a^D\big)$ inverse.
\end{enumerate}
\begin{proof} $(1)\Leftrightarrow (2)$ Let $n=ind(a)$. Then $a^{n+1}a^D=a^n$ and $a(a^D)^2=a^D$. We directly verify that
$a^{\tiny\textcircled{\#}}=a^{(a^*,a^*)}$ if and only if $a^{\tiny\textcircled{\#}}=a^{\big((a^D)^*,(a^D)^*\big)}$.
This equivalence is proved by Lemma 3.5.

$(2)\Rightarrow (3)$ This is obvious by Corollary 2.6.

$(3)\Rightarrow (1)$ By hypothesis, we have $a^{\tiny\textcircled{\#}}=a_{lh}^{(a^D,(a^D)^*)}.$
By virtue of Theorem 2.1, $\ell(a)=\ell(a^D)^*$. Since $(1-a^Da)^*(a^D)^*=0$, we deduce that
$(1-a^Da)^*a=0$, and so $a=a^*(a^D)^*a$. This implies that $aR\subseteq a^*R$.
Therefore $a\in R$ is EP by ~\cite[Corollary 3.10]{XS1}.

$(1)\Leftrightarrow (4)\Leftrightarrow (5)$ These can be proved in a similar way.\end{proof}

\section{hybrid core-EP-$(b,c)$-inverses}

The objective of this section is to establish the fundamental properties of the hybrid core-EP-$(b,c)$-inverses within a ring.

\begin{thm} Let $a,b,c\in R$. Then the following are equivalent:\end{thm}
\begin{enumerate}
\item [(1)] $a\in R_{lh(b,c)}^{\tiny\textcircled{D}}$.
\vspace{-.5mm}
\item [(2)] $a\in R^{\tiny\textcircled{D}}$ and $aa^{\tiny\textcircled{D}}a\in R_{lh(b,c)}^{\tiny\textcircled{\#}}.$
\end{enumerate}
In this case, $a^{\tiny\textcircled{D}}=(aa^{\tiny\textcircled{D}}a)_{lh}^{(b,c)}.$
\begin{proof} $(1)\Rightarrow (2)$ By hypothesis, there exist $x,y\in R$ such that $$a=x+y, x^*y=yx=0, x\in
R_{lh(b,c)}^{\tiny\textcircled{\#}}, y\in R^{nil}.$$ Then $x\in R^{\tiny\textcircled{\#}}$.
By virtue of~\cite[Theorem 2.1]{CM5}, $a\in R^{\tiny\textcircled{D}}$ and $$
a^{\tiny\textcircled{D}}=x^{\tiny\textcircled{\#}}=x_{lh}^{(b,c)}=(aa^{\tiny\textcircled{D}}a)_{lh}^{(b,c)}.$$

$(2)\Rightarrow (1)$ By hypothesis, $a\in R^{\tiny\textcircled{D}}$ and $aa^{\tiny\textcircled{D}}a\in R_{lh(b,c)}^{\tiny\textcircled{\#}}.$ One easily verifies that $$\big((aa^{\tiny\textcircled{D}}a)_{lh}^{(b,c)}\big)^{\tiny\textcircled{\#}}=a^{\tiny\textcircled{D}}.$$
In view of ~\cite[Theorem 2.1]{CM5}, there exist $x,y\in R$ such that $$a=x+y, x^*y=yx=0, x\in
R^{\tiny\textcircled{\#}}, y\in R^{nil}.$$ Moreover, we have $$
x^{\tiny\textcircled{\#}}=a^{\tiny\textcircled{D}}=(aa^{\tiny\textcircled{D}}a)_{lh}^{(b,c)}=x_{lh}^{(b,c)}.
$$ Therefore $x\in R_{lh(b,c)}^{\tiny\textcircled{\#}}$, as desired.\end{proof}

\begin{cor} Let $a,b,c\in R$. Then the following are equivalent:\end{cor}
\begin{enumerate}
\item [(1)] $a\in R_{lh(b,c)}^{\tiny\textcircled{D}}$.
\vspace{-.5mm}
\item [(2)] There exist $x\in R$ and $n\in {\Bbb N}$ such that $$x=ax^2, (ax)^*=ax, axa\in R_{lh(b,c)}^{\tiny\textcircled{\#}}, a^n=xa^{n+1}$$
\vspace{-.5mm}
\item [(3)] There exist $x\in R$ and $n\in {\Bbb N}$ such that $$x=ax^2, (ax)^*=ax, \ell(ax)=\ell(b), Rx=Rc,
a^n=xa^{n+1}.$$
\end{enumerate}
In this case, $x=a_{lh(b,c)}^{\tiny\textcircled{D}}$.
\begin{proof} $(1)\Rightarrow (2)$ Set $x=a^{\tiny\textcircled{D}}$. Then $x=ax^2, (ax)^*=ax~\mbox{and}~a^n=xa^{n+1}$ for some $n\in {\Bbb N}$. According to Theorem 4.1, $axa\in R_{lh(b,c)}^{\tiny\textcircled{\#}}$, as required.

$(2)\Rightarrow (3)$ By hypothesis, we have
$$x=ax^2, (ax)^*=ax, axa\in R_{lh(b,c)}^{\tiny\textcircled{\#}}, a^n=xa^{n+1}$$ for some $n\in {\Bbb N}$.
Then $\ell(axa)=\ell(b)$ and $R(axa)^*=Rc$. Obviously, $x=xax$, and then $\ell(axa)=\ell(ax)$.
We easily check that $(axa)^*=a^*(ax)^*=a^*ax\in Rx$. On the other hand,
$x=xax=x(ax)^*=x(axax)^*=xx^*(axa)^*\in R(axa)^*$.
Therefore $R(axa)^*=Rx$, as required.

$(3)\Rightarrow (1)$ By hypothesis, we have $a^{\tiny\textcircled{D}}=x$. Hence $x=xax$ by ~\cite[Theorem 2.10]{GC}.
Since $\ell(ax)=\ell(b)$, we have $\ell(axa)=\ell(ax)=\ell(b)$.
On the other hand, $R(axa)^*=Ra^*ax\subseteq Rx$. As $x=xax=x(axax)^*=xx^*(axa)^*$, we have
$Rx\subseteq R(axa)^*$. Therefore $R(axa)^*=Rx=Rc$. Clearly, $(axa)^{\tiny\textcircled{D}}=x$.
Therefore $axa\in R_{lh(b,c)}^{\tiny\textcircled{\#}}$ by Theorem 2.1. According to Theorem 4.1, $a\in R_{lh(b,c)}^{\tiny\textcircled{D}}$.\end{proof}

We next investigate polar-like properties for the generalized left hybrid core-$(b,c)$-invertibility. We are ready to prove:

\begin{thm} Let $a,b,c\in R$. Then $a\in R_{lh(b,c)}^{\tiny\textcircled{D}}$ if and only if\end{thm}
\begin{enumerate}
\item [(1)] $a\in R^D$;
\item [(2)] there exists a projection $p\in R$ such that $$a+p\in R^{-1}, 1-p\in R_{lh(b,c)}^{\tiny\textcircled{\#}}, pa=pap\in R^{nil}.$$
\end{enumerate}
\begin{proof} $(1)\Rightarrow (2)$ In view of ~\cite[Theorem 2.3]{GC}, $a\in R^D$. Since $a\in R_{lh(b,c)}^{\tiny\textcircled{D}}$, by using Theorem 4.1, there exist $x,y\in R$ such that $$a=x+y, x^*y=yx=0, x\in
R_{lh(b,c)}^{\tiny\textcircled{\#}}, y\in R^{nil}.$$
By virtue of Theorem 2.1, we have
$$\begin{array}{c}
x=xx^{\tiny\textcircled{\#}}x=xx_{lh(b,c)}^{\tiny\textcircled{\#}}x,\\
\ell(x)=\ell(b),\\
Rx^*=Rc.
\end{array}$$ Let $p=1-xx_{lh(b,c)}^{\tiny\textcircled{\#}}$.
Then $p^2=p=p^*$ and $px=0$. We directly check that
$$\begin{array}{c}
(x+1-xx_{lh(b,c)}^{\tiny\textcircled{\#}})(x_{lh(b,c)}^{\tiny\textcircled{\#}}+1-x_{lh(b,c)}^{\tiny\textcircled{\#}}x)=1\\
=(x_{lh(b,c)}^{\tiny\textcircled{\#}}+1-x_{lh(b,c)}^{\tiny\textcircled{\#}}x)(x+1-xx_{lh(b,c)}^{\tiny\textcircled{\#}}).
\end{array}$$
Hence, $$(x+p)^{-1}=x_{lh(b,c)}^{\tiny\textcircled{\#}}+1-x_{lh(b,c)}^{\tiny\textcircled{\#}}x.$$ Since $y(x+p)=y(x+1-xx_{lh(b,c)}^{\tiny\textcircled{\#}})=y$, we see that $y(x+p)^{-1}=y\in R^{nil}$.
Then $(x+p)^{-1}y\in R^{nil}$. Hence, $1+(x+p)^{-1}y\in R^{-1}$.
Therefore, we check that
$$\begin{array}{rll}
pa&=&p(x+y)=py=(1-xx_{lh(b,c)}^{\tiny\textcircled{\#}})y\\
&=&(1-xx_{lh(b,c)}^{\tiny\textcircled{\#}})^*y\\
&=&(1-(x_{lh(b,c)}^{\tiny\textcircled{\#}})^*x^*)y\\
&=&y\in R^{nil},\\
pa(1-p)&=&yxx_{lh(b,c)}^{\tiny\textcircled{\#}}=0, pa=pap,\\
a+p&=&x+y+p=(x+p)[1+(x+p)^{-1}y]\in R^{-1}.\end{array}$$

Obviously, $(1-p)^*=(1-p)^2=1-p\in R^{\tiny\textcircled{\#}}.$ Moreover, we check that

Claim 1. $\ell(1-p)=\ell(b)$.

$\ell(1-p)=\ell(xx_{lh}^{\tiny\textcircled{\#}})=\ell(x)=\ell(b)$.

Claim 2. $r(1-p)^*=r(c)$.

$r(1-p)=r(xx_{lh}^{\tiny\textcircled{\#}})=r(x^*)=r(c)$.

Therefore $1-p\in R_{lh(b,c)}^{\tiny\textcircled{\#}}$, as required.

$(2)\Rightarrow (1)$ By hypothesis, there exists a projection $p\in R$ such that $$u:=a+p\in R^{-1}, 1-p\in R_{lh(b,c)}^{\tiny\textcircled{\#}}, pa=pap\in R^{nil}.$$
Set $x=(1-p)a$ and $y=pa$. Then $a=x+y$ and $$\begin{array}{rll}
x^*y&=&[(1-p)a]^*(pa)=[a^*(1-p)^*]p^*a=a^*(1-p)^*p^*a=0,\\
yx&=&pa(1-p)a=(pap)(1-p)a=0,\\
y&=&pa\in R^{nil}.
\end{array}$$
Step 1. $x\in R^D$.

Since $pa(1-p)=0$, we have $a=pap+(1-p)ap+(1-p)a(1-p)$, i.e.,
$a=\left(
\begin{array}{cc}
pap&0\\
(1-p)ap&(1-p)a(1-p)
\end{array}
\right)_p$. By hypothesis, $pa\in R^{nil}$, then $pap\in R^{nil}$.
Hence, $pap\in \big(pRp\big)^{nil}\subseteq \big(pRp\big)^D.$ In view of ~\cite[Theorem 2.3]{CK},
$(1-p)a(1-p)\in \big((1-p)R(1-p)\big)^D\subseteq R^D.$ By using Cline's formula again,
$x=(1-p)a\in R^D,$ as desired.

Step 2. $x\in R^{\#}$. Observing that $$\begin{array}{rll}
u^{-1}x^2&=&u^{-1}(1-p)a(1-p)a=u^{-1}a(1-p)a-u^{-1}[pa(1-p)]a\\
&=&u^{-1}a(1-p)a=(a+p)^{-1}(a+p)(1-p)a=(1-p)a=x,
\end{array}$$ Set $k=ind(x)$. Then
$xx^Dx=(u^{-k}x^{k+1})x^Dx=u^{-k}(x^{k+1}x^D)x=u^{-k}x^{k+1}=x.$ This implies that
$x\in R^{\#}$ and $x^{\#}=x^D$.

Step 3. $x\in R_{lh(b,c)}^{\tiny\textcircled{\#}}$.

Clearly, $xu^{-1}=(1-p)a(a+p)^{-1}=(1-p)(a+p)(a+p)^{-1}=1-p$. Then $xu^{-1}x=(1-p)x=x$ and $(xu^{-1})^*=(1-p)^*=1-p^*=1-p=xu^{-1}$. Hence, $x\in R^{(1,3)}$. According to ~\cite[Theorem 2.6]{X1}, $x\in R^{\tiny\textcircled{\#}}$.

Claim 1. $\ell(x)=\ell(b)$.

$\ell(x)=\ell[(1-p)a]=\ell[(1-p)u]=\ell(1-p)=\ell(b).$

Claim 2. $r(x^*)=r(c)$.

$r(x^*)=r[a^*(1-p)]=r[u^*(1-p)]=r(1-p)=r(c)$.

In light of Corollary 2.2, $x\in R_{lh(b,c)}^{\tiny\textcircled{\#}}.$

Therefore $a\in R_{(b,c)}^{\tiny\textcircled{D}}$.\end{proof}

\begin{cor} Let $a,b,c\in R$. Then $a\in R_{lh(b,c)}^{\tiny\textcircled{D}}$ if and only if the following two conditions hold:\end{cor}
\begin{enumerate}
\item [(1)] $a\in R^D$;
\item [(2)] there exists a projection $p\in R$ such that $$a+p\in R^{-1}, \ell(1-p)=\ell(b), r(1-p)=r(c), pa=pap\in R^{nil}.$$
\end{enumerate}
\begin{proof} By using Theorem 4.3, we obtain the result, as for a projection $p$, $1-p\in R_{lh(b,c)}^{\tiny\textcircled{\#}}$ if and only if
$\ell(1-p)=\ell(b), r(1-p)=r(c)$.\end{proof}

\begin{exam}\end{exam}Let $A=\left(
  \begin{array}{cccc}
    0&1&-1&0\\
    0&-1&0&0\\
    0&0&0&0\\
    -1&-1&0&0\\
  \end{array}
\right)\in {\Bbb C}^{4\times 4}$. We take the involution on ${\Bbb C}^{4\times 4}$ as the conjugate transpose. Then $$A^D=\left(
  \begin{array}{cccc}
    0&1&0&0\\
    0&-1&0&0\\
    0&0&0&0\\
    0&0&0&0\\
  \end{array}
\right), (A^D)^{\tiny\textcircled{\#}}=\left(
  \begin{array}{cccc}
    -\frac{1}{2}&\frac{1}{2}&0&0\\
   \frac{1}{2}& -\frac{1}{2}&0&0\\
    0&0&0&0\\
    0&0&0&0\\
  \end{array}
\right).$$ In view of ~\cite[Corollary 4.2]{CM5}, we compute that $$A^{\tiny\textcircled{D}}=(A^D)^2(A^D)^{\tiny\textcircled{\#}}=\left(
  \begin{array}{cccc}
    -\frac{1}{2}&\frac{1}{2}&0&0\\
   \frac{1}{2}& -\frac{1}{2}&0&0\\
    0&0&0&0\\
    0&0&0&0\\
  \end{array}
\right).$$ Choose $$B=AA^{\tiny\textcircled{D}}A=\left(
  \begin{array}{cccc}
    0&1&-\frac{1}{2}&0\\
   0&-1&\frac{1}{2}&0\\
    0&0&0&0\\
    0&0&0&0\\
  \end{array}
\right), C=B^*=\left(
  \begin{array}{cccc}
    0&0&0&0\\
   1&-1&0&0\\
   -\frac{1}{2}&\frac{1}{2}&0&0\\
    0&0&0&0\\
  \end{array}
\right).$$ Then $A$ has left hybrid core-EP-$(B,C)$-inverse and
$$A_{lh(B,C)}^{\tiny\textcircled{D}}=\left(
  \begin{array}{cccc}
    -\frac{1}{2}&\frac{1}{2}&0&0\\
   \frac{1}{2}& -\frac{1}{2}&0&0\\
    0&0&0&0\\
    0&0&0&0\\
  \end{array}
\right).$$
Choose $P=I_4-BB^{\tiny\textcircled{D}}=\left(
  \begin{array}{cccc}
    \frac{1}{2}&\frac{1}{2}&0&0\\
   \frac{1}{2}& \frac{1}{2}&0&0\\
    0&0&1&0\\
    0&0&0&1\\
  \end{array}
\right).$ We directly verify that
$$\begin{array}{rll}
A+P&=&\left(
  \begin{array}{cccc}
    \frac{1}{2}&\frac{3}{2}&-1&0\\
   \frac{1}{2}& -\frac{1}{2}&0&0\\
    0&0&1&0\\
    -1&-1&0&1\\
  \end{array}
\right)~\mbox{is invertible},\\
\ell(I_4-P)&=&\ell(B), r(I_4-P)=r(C),\\
PA&=&\left(
  \begin{array}{cccc}
    0&0&-\frac{1}{2}&0\\
   0&0&-\frac{1}{2}&0\\
    0&0&0&0\\
    -1&-1&0&0\\
  \end{array}
\right)=PAP~\mbox{is nil}.
\end{array}$$

\section{representations of hybrid core-EP-$(b,c)$-inverses}

The aim of this section is to establish representations of the hybrid core-EP-$(b,c)$-inverse by using the hybrid core-$(b,c)$-inverse.
We come now to the demonstration for which this section has been developed.

\begin{thm} Let $a,b,c\in R$. Then the following are equivalent:\end{thm}
\begin{enumerate}
\item [(1)] $a\in R_{lh(b,c)}^{\tiny\textcircled{D}}.$
\vspace{-.5mm}
\item [(2)] There exists $m\in {\Bbb N}$ such that $a^k\in R_{lh(b,c)}^{\tiny\textcircled{\#}}$ for any $k\geq m$.
\vspace{-.5mm}
\item [(3)] $a^m\in R_{lh(b,c)}^{\tiny\textcircled{\#}}$ for some $m\in {\Bbb N}$.
\end{enumerate}
In this case, $$a_{lh(b,c)}^{\tiny\textcircled{D}}=a^{m-1}(a^m)_{lh(b,c)}^{\tiny\textcircled{\#}}.$$
\begin{proof} $(1)\Rightarrow (2)$ In view of Theorem 4.1, $a\in R^{\tiny\textcircled{D}}$ and $a^{\tiny\textcircled{D}}=(aa^{\tiny\textcircled{D}}a)_{lh(b,c)}^{\tiny\textcircled{\#}}$.
Set $m=ind(a)$ and $k\geq m$. Then $a^k\in R^{\tiny\textcircled{\#}}$ and $a^{\tiny\textcircled{D}}=a^{k-1}(a^k)^{\tiny\textcircled{\#}}$.

$$\begin{array}{rll}
\ell(aa^{\tiny\textcircled{D}}a)&=&\ell(b),\\
R(aa^{\tiny\textcircled{D}}a)^*&=&Rc.
\end{array}$$

If $xa^k=0$, then $xaa^{\tiny\textcircled{D}}a=0$, and so $xb=0$. If $xb=0$, then $xaa^{\tiny\textcircled{D}}a=0$.
This implies that $xaa^D(a^k)(a^k)^{(1,3)}a=0$, and so $xaa^Da^k=0$. This implies that $xa^k=0$.
Thus $\ell(a^k)=\ell(b).$

We easily verify that
$$a^k=aa^Da^k=a(a^Da^k(a^k)^{(1,3)})a^k=(aa^{\tiny\textcircled{D}}a)a^{k-1}.$$ Hence,
$(a^k)^*=(a^{k-1})^*(aa^{\tiny\textcircled{D}}a)^*\in Rc.$

We have $aa^{\tiny\textcircled{D}}a=aa^Da^k(a^k)^{(1,3)}a$, and then
$(aa^{\tiny\textcircled{D}}a)^*=(aa^D(a^k)^{(1,3)}a)^*(a^k)^*$. This implies that
$c\in R(a^k)^*$; hence, $R(a^k)^*=Rc$. Thus $(a^k)^{\tiny\textcircled{\#}}=(a^k)_{lh}^{(b,c)}$.
Therefore $a^k\in R_{lh(b,c)}^{\tiny\textcircled{\#}}$, as desired.

$(2)\Rightarrow (3)$ This is obvious.

$(3)\Rightarrow (1)$ By hypothesis, $a^m\in R_{b,c}^{\tiny\textcircled{\#}}$ for some $m\in {\Bbb N}$.
Then $a^m\in R^{\tiny\textcircled{\#}}$ and $(a^m)^{\tiny\textcircled{\#}}=(a^m)_{lh}^{(b,c)}$.
Thus we have $\ell(a^m)=\ell(b)$ and $R(a^m)^*=Rc$. In light of~\cite[Theorem 2.5]{GC}, $a\in R^{\tiny\textcircled{D}}$ and $a^{\tiny\textcircled{D}}=a^{m-1}(a^m)^{\tiny\textcircled{\#}}.$ Hence $aa^{\tiny\textcircled{D}}a=a^m(a^m)^{\tiny\textcircled{\#}}a$.

Claim 1. $\ell[aa^{\tiny\textcircled{D}}a]=\ell(b)$.

If $xa^m(a^m)^{\tiny\textcircled{\#}}a=0$, then $xa^m=0$, hence, $xb=0$. If $xb=0$, then
$xa^m=0$, and so $xa^m(a^m)^{\tiny\textcircled{\#}}a=0$. This implies that $\ell[aa^{\tiny\textcircled{D}}a]=\ell(b)$.

Claim 2. $R[aa^{\tiny\textcircled{D}}a]^*=Rc$.

Clearly, $a^m=[a^m(a^m)^{\tiny\textcircled{\#}}a]a^{m-1}$; hence, $$c\in R(a^m)^*\subseteq R(a^{m-1})^*[a^m(a^m)^{\tiny\textcircled{\#}}a]^*\subseteq
R[a^m(a^m)^{\tiny\textcircled{\#}}a]^*.$$ On the other hand, $[a^m(a^m)^{\tiny\textcircled{\#}}a]^*=[(a^m)^{\tiny\textcircled{\#}}a]^*(a^m)^*\in R(a^m)^*\subseteq Rc$. Thus, $R[aa^{\tiny\textcircled{D}}a]^*=Rc$.

Therefore we have $aa^{\tiny\textcircled{D}}a\in \mathcal{A}_{(b,c)}^{\tiny\textcircled{\#}}$. According to Theorem 4.1, we obtain the result.\end{proof}

\begin{cor} Let $a,b,c\in R$. Then the following are equivalent:\end{cor}
\begin{enumerate}
\item [(1)] $a\in R_{lh(b,c)}^{\tiny\textcircled{D}}$.
\vspace{-.5mm}
\item [(2)] $a\in R^D$ and $a^D\in R_{lh(b,c)}^{\tiny\textcircled{\#}}$.
\end{enumerate}
In this case, $$a_{lh(b,c)}^{\tiny\textcircled{D}}=(a^D)^2(a^D)_{lh(b,c)}^{\tiny\textcircled{\#}}.$$
\begin{proof} $(1)\Rightarrow (2)$ Set $m=ind(a)$. In view of Theorem 5.1, $a^m\in R^{\#}$, and then $a\in R^D$.
By using Theorem 5.1 again, $a^m\in R_{lh(b,c)}^{\tiny\textcircled{\#}}$ and
$a^{\tiny\textcircled{D}}=a^{m-1}(a^m)^{\tiny\textcircled{\#}}$. In view of~\cite[Theorem 2.7]{GC}, $a^D\in \mathcal{A}^{\tiny\textcircled{\#}}$.

By virtue of Theorem 2.1, we have $\ell(a^m)=\ell(b)$ and $R(a^m)^*=Rc$.

Claim 1. $\ell(a^D)=\ell(b)$.

If $xa^D=0$, then $xa^m=(xa^D)a^{m+1}=0$, and so $x\in \ell(b)$. If $xb=0$, then $xa^m=0$, hence, $xa^D=(xa^m)(a^D)^{m+1}=0$.
Thus, $\ell(a^D)=\ell(b)$.

Claim 2. $R(a^D)^*=Rc$.

Clearly, $a^D=a^m(aa^D)$, hence, $(a^D)^*=(aa^D)^*(a^m)^*\in R(a^m)^*\subseteq Rc$.
On the other hand, $a^m=a^Da^{m+1}$; hence, $(a^m)^*\in R(a^D)^*$. This implies that $c\in R(a^D)^*$.
Hence, $R(a^D)^*=Rc$.

Therefore $a^D\in R_{lh(b,c)}^{\tiny\textcircled{\#}}$, as required.

$(2)\Rightarrow (1)$ By hypothesis, we have $a^D\in R_{lh}^{\tiny\textcircled{\#}}$. In view of ~\cite[Theorem 4.1]{CM5},
$a\in \mathcal{A}^{\tiny\textcircled{D}}$. Moreover, we have $\ell(a^D)=\ell(b)$ and $R(a^D)^*=Rc$.

Claim 1. $\ell(aa^{\tiny\textcircled{D}}a)=\ell(b)$.

If $xaa^{\tiny\textcircled{D}}a=0$, we have $xa^D=0$, and then $xb=0$.
If $xb=0$, then $xa^D=0$, and so $x[aa^{\tiny\textcircled{D}}a]=(xa^D)a^2a^{\tiny\textcircled{D}}a=0$.
Thus $\ell(aa^{\tiny\textcircled{D}}a)=\ell(b)$.

Claim 2. $R(aa^{\tiny\textcircled{D}}a)^*=Rc$.

Obviously, $(aa^{\tiny\textcircled{D}}a)^*=a^*(a^Da^2a^{\tiny\textcircled{D}})^*\in R(a^D)^*\in Rc$.
On the other hand, $a^D=[aa^{\tiny\textcircled{D}}a]a^{m-1}(a^D)^{m}$; whence, $$(a^D)^*=[a^{m-1}(a^D)^{m}]^*[aa^{\tiny\textcircled{D}}a]^*\in R[aa^{\tiny\textcircled{D}}a]^*.$$ Therefore $R(aa^{\tiny\textcircled{D}}a)^*=Rc$.

Accordingly, $aa^{\tiny\textcircled{D}}a\in R_{lh(b,c)}^{\tiny\textcircled{\#}}$.
By virtue of Theorem 4.1, $a\in R_{lh(b,c)}^{\tiny\textcircled{D}}$, as asserted.\end{proof}

\begin{cor} Let $a,b,c\in R$. Then the following hold:\end{cor}
\begin{enumerate}
\item [(1)] $a\in R^{\tiny\textcircled{\#}}$ if and only if $a\in R_{lh(a, a^*)}^{\tiny\textcircled{\#}}$.
\vspace{-.5mm}
\item [(2)] $a\in R^{\tiny\textcircled{D}}$ if and only if $a\in R^{D}$ and $a\in R_{lh(a^D, (a^D)^*)}^{\tiny\textcircled{D}}$.
\end{enumerate}
\begin{proof} This is obvious by Corollary 5.2.\end{proof}

Let ${\Bbb C}^{n\times n}$ be the ring of all $n\times n$ complex matrices, with conjugate transpose as the involution. For a complex $A\in {\Bbb C}^{n\times n}$, it follows by Corollary 5.3 that the core-EP inverse and left hybrid core-EP-$(A^D,(A^D)^*)$-inverse coincide with each other and
$A^{\tiny\textcircled{D}}=A_{lh(A^D,(A^D)^*)}^{\tiny\textcircled{D}}$.

\begin{cor} Let $a,b\in R$. Then the following are equivalent:\end{cor}
\begin{enumerate}
\item [(1)] $a\in R_{lh(b,b)}^{\tiny\textcircled{D}}$.
\item [(2)] $a\in R_{h(b,b)}^{\tiny\textcircled{D}}$.
\item [(3)] $a\in R^D, a^Db\in R$ is EP, $r(a^Db)=r(b), (ba^D)^{\pi}a^D=0$ and $(a^Db)^{\pi}b=0$.
\end{enumerate}
\begin{proof} This is obvious by Corollary 3.4 and Corollary 5.2.\end{proof}

\begin{thm} Let $a,b,c\in R$. Then the following are equivalent:\end{thm}
\begin{enumerate}
\item [(1)] $a\in R$ is *-DMP.
\item [(2)] $a\in R_{lh((a^D)^*,a^D)}^{\tiny\textcircled{D}}$.
\item [(3)] $a\in R_{rh((a^D)^*,a^D)}^{\tiny\textcircled{D}}$.
\end{enumerate}
\begin{proof} $(1)\Rightarrow (2)$ In view of~\cite[Lemma 3.1]{GC1}, we have
$a\in R^{\tiny\textcircled{D}}$ and $a^{\tiny\textcircled{D}}=a^D$. We directly check that
$$\begin{array}{rll}
a^{\tiny\textcircled{D}}(aa^{\tiny\textcircled{D}}a)a^{\tiny\textcircled{D}}&=&a^{\tiny\textcircled{D}},\\
\ell(aa^{\tiny\textcircled{D}}a)&=&\ell(a^D)^*,\\
R(aa^{\tiny\textcircled{D}}a)^*&=&Ra^D.
\end{array}$$ Hence, $a^{\tiny\textcircled{D}}=(aa^{\tiny\textcircled{D}}a)_{lh}^{{\tiny\textcircled{\#}}_{((a^D)^*,a^D)}}.$
Therefore $a\in R_{lh((a^D)^*,a^D)}^{\tiny\textcircled{D}}$ by Theorem 4.1.

$(2)\Rightarrow (1)$ By hypothesis, we have $$R(aa^{\tiny\textcircled{D}}a)^*=Ra^D.$$
Since $(aa^{\tiny\textcircled{D}}a)^*[1-aa^{\tiny\textcircled{D}}]=a^*(aa^{\tiny\textcircled{D}})^*[1-aa^{\tiny\textcircled{D}}]
=a^*(aa^{\tiny\textcircled{D}})[1-aa^{\tiny\textcircled{D}}]=0$, we deduce that $a^D[1-aa^{\tiny\textcircled{D}}]=0$; hence,
$a^D=a^Daa^{\tiny\textcircled{D}}$. Thus, we have $$\begin{array}{rll}
a^{\tiny\textcircled{D}}-a^D&=&a^n(a^{\tiny\textcircled{D}})^{n+1}-a^Daa^{\tiny\textcircled{D}}\\
&=&[a^n-a^Da^{n+1}](a^{\tiny\textcircled{D}})^{n+1}.
\end{array}$$ This implies that $a^{\tiny\textcircled{D}}=a^D$. According to ~\cite[Lemma 3.1]{GC1},
$a\in R$ is *-DMP, as asserted.

$(1)\Leftrightarrow (3)$ This is proved by a similar way.\end{proof}

\begin{cor} Let $a\in R$. Then the following are equivalent:\end{cor}
\begin{enumerate}
\item [(1)] $a\in R$ is *-DMP.
\item [(2)] $a\in R_{((a^D)^*,a^D)}^{\tiny\textcircled{D}}$.
\end{enumerate}
\begin{proof} Since $a^D=a^Daa^D$, $a\in R_{((a^D)^*,a^D)}^{\tiny\textcircled{D}}$ if and only if $a\in R_{h((a^D)^*,a^D)}^{\tiny\textcircled{D}}$.
Therefore we complete the proof by Theorem 5.5.\end{proof}


\end{document}